\documentclass{amsart}
\usepackage[utf8]{inputenc}
\usepackage{amssymb}
\usepackage{enumerate}
\usepackage[colorinlistoftodos]{todonotes}

\usepackage{amsthm, url}
\usepackage{xcolor}
\usepackage{tikz,pgf}
\usetikzlibrary{decorations.text}

\newtheorem{theorem}{Theorem}[section]
\newtheorem*{fibering lemma}{Fibering Lemma}
\newtheorem*{decomposition lemma}{Decomposition Lemma}
\newtheorem*{hurewicz theorem}{Hurewicz Theorem}
\newtheorem{lemma}[theorem]{Lemma}
\newtheorem{proposition}[theorem]{Proposition}
\newtheorem{corollary}[theorem]{Corollary}

\theoremstyle{definition}
\newtheorem{definition}[theorem]{Definition}
\newtheorem{comm}[theorem]{Comment}
\newtheorem{example}[theorem]{Example}

\newtheorem{question}[theorem]{Question}

\newcommand{\U}[1]{{\mathcal{U}_{#1}}}
\newcommand{\V}[1]{{\mathcal{V}_{#1}}}
\newcommand{\W}[1]{{\mathcal{W}_{#1}}}
\newcommand{\df}[1]{{{\bf #1}}}

\DeclareMathOperator{\as}{asdim}
\DeclareMathOperator{\con}{con}
\DeclareMathOperator{\mesh}{mesh}
\DeclareMathOperator{\diam}{diam}
\DeclareMathOperator{\dist}{dist}

\DeclareMathOperator{\ord}{ord}

\begin{document}

\bibliographystyle{abbrv}

\title[On product stability of asymptotic property~C]{On stability of asymptotic property~C for products and some group extensions}
\author[G.~Bell]{G.~Bell}
\address{Department of Mathematics and Statistics, The University of North Carolina at Greensboro, Greensboro, NC 27402, USA} 
\email{gcbell@uncg.edu}

\author[A.~Nag\'orko]{A.~Nag\'orko}

\address{Faculty of Mathematics, Informatics, and Mechanics, University of Warsaw, Banacha 2, 02-097 Warszawa, Poland
}
\email{amn@mimuw.edu.pl}

\begin{abstract}
We show that Dranishnikov's asymptotic property~C is preserved by  direct products and the free product of discrete metric spaces. In particular, if $G$ and $H$ are groups with asymptotic property~C, then both $G \times H$ and $G * H$ have asymptotic property~C. We also prove that a group~$G$ has asymptotic property~C if $1\to K\to G\to H\to 1$ is exact, if $\as K<\infty$, and if $H$ has asymptotic property~C. The groups are assumed to have left-invariant proper metrics and need not be finitely generated. These results settle questions of Dydak and Virk~\cite{dydak-virk2016}, of Bell and Moran~\cite{bell-moran2015}, and an open problem in topology in \cite{openproblems}.
\end{abstract}
\thanks{The second author was supported by the NCN (Narodowe Centrum Nauki) grant no. 2011/01/D/ST1/04144.}

\maketitle
\section{Introduction}

Asymptotic property~C is a coarse invariant of metric spaces introduced by A.~N.~Dranishnikov in 2000~\cite{dranish2000}.
Recently it received a lot of attention~\cite{beckhardt2015, beckhardt-goldfarb2016, bell-moran2015, dranishnikov-zarichnyi2014, dydak2016, dydak-virk2016, yamauchi2015}.

Asymptotic property~C is a coarse analog of topological property~C~\cite{haver}. Topological property~C fails to be preserved by direct products in very striking ways: there is an example of a separable metrizable space $X$ with property~C whose product with the irrationals fails to have property~C~\cite{EPol1986} (under the Continuum Hypothesis); there is also a separable complete metric space $X$ with property~C whose square $X\times X$ fails to have property~C~\cite{Pol-Pol2009}.
The question whether asymptotic property~C is preserved by direct products is widely known and was published in~\cite{openproblems}.
We answer it in the affirmative in the following theorem.

\begin{theorem}\label{thm:direct product}
  Let $X$ and $Y$ be metric spaces.
  If $X$ and $Y$ have asymptotic property~C, then $X \times Y$ has asymptotic property~C.
\end{theorem}

Theorem~\ref{thm:direct product} is an analog of a theorem of D. M. Rohm that states that topological property~C is preserved by a direct product if one of the factors is compact~\cite[Theorem 3]{rohm1990}.
Asymptotic property~C has a compactness-like property built into the definition, which requires that the sequence of constructed families is finite.

The technique used to prove Theorem \ref{thm:direct product} can also be used to prove a useful fibering lemma:

\begin{fibering lemma}
  If $f \colon X \to Y$ is 
  a uniformly expansive map of metric spaces, 
  $Y$ has asymptotic property~C, and 
  coarse fibers of~$f$ have uniform asymptotic property~C,
  then $X$ has asymptotic property~C.
\end{fibering lemma}

As an immediate consequence of this lemma, we can provide an affirmative answer to a question of J.~Dydak and \v{Z}.~Virk \cite{dydak-virk2016} regarding preservation of asymptotic property~C via maps with finite asymptotic dimension.

The Fibering Lemma is effective in settling various questions concerning asymptotic property~C of countable groups or finitely generated groups. Consider the following short exact sequences:
\[
  1 \to G \to G \times H \to H \to 1
\]
and
\[
  1 \to [[G, H]] \to G * H \to G \times H \to 1
\]
with groups $G$ and $H$ endowed with left-invariant proper metrics.
In the present paper we show that the surjective maps $G \times H \to H$ and $G * H \to G \times H$ satisfy the assumptions of the Fibering Lemma if $G$ and $H$ have asymptotic property~C. In the latter sequence $[[G, H]]$ is the commutator subgroup of $G * H$, which is quasi-isometric to a tree~\cite{serre}, hence it is of finite asymptotic dimension. Therefore the last exact sequence is a special case of the group extension
\[
  1 \to K \to G \to H \to 1,
\]
which (as we show) satisfies assumptions of the Fibering Lemma if $K$ has finite asymptotic dimension and $H$ has asymptotic property~C. These observations are summarized in Theorem~\ref{thm:permanence for groups}.

\begin{theorem}\label{thm:permanence for groups}
  Let $G$ and $H$ be groups with proper left-invariant metrics.
  \begin{enumerate}
  \item If $G$ and $H$ have asymptotic property~C, then $G \times H$ has asymptotic property~C.
  \item If $G$ and $H$ have asymptotic property~C, then $G * H$ has asymptotic property~C.
  \item If $1 \to K \to G \to H \to 1$ is an exact sequence, $H$ has asymptotic property~C and $K$ has finite asymptotic dimension, then $G$ has asymptotic property~C.
  \end{enumerate}
\end{theorem}

Item (2) in Theorem \ref{thm:permanence for groups} settles the question of whether asymptotic property~C is preserved by free products, which was posed in \cite{bell-moran2015}. We do not consider the more general question of free products with an amalgamated subgroup.

The result about free products of groups can be extended to free products of metric spaces.

\begin{theorem}\label{thm:free product}
  Let $X$ and $Y$ be metric spaces that are discrete in the metric sense (see Definition \ref{def:discrete}).
  If $X$ and $Y$ have asymptotic property~C, then $X * Y$ has asymptotic property~C.
\end{theorem}

The proof differs from the proof given in the case of groups; it relies on the following Decomposition Lemma and is of independent interest.
\begin{decomposition lemma}
  Let $X$ be a metric space. 
  Assume that there exists $k$ such that for each (infinite) sequence $R_1, R_2, \ldots$ of real numbers there exists a finite sequence $\U1, \U2, \ldots, \U{n}$ of families of subsets of $X$ such that
  \begin{enumerate}
  \item $\bigcup_i \U{i}$ covers $X$,
  \item $\U{i}$ is $R_i$-disjoint,
  \item $\U{i}$ has asymptotic dimension bounded by $k$ uniformly.
  \end{enumerate}
  Then $X$ has asymptotic property~C.
\end{decomposition lemma}

The authors would like to thank the anonymous referee for many helpful suggestions and improvements. The authors have also been made aware that Theorem \ref{thm:direct product} was shown independently in \cite{Davila}.


\section{Preliminaries}\label{sec:preliminaries}

Let $U$ and $V$ be subsets of a metric space $X$. 
Let $R>0$ be a real number. The subsets $U$ and $V$ are said to be \df{$R$-disjoint} if $\dist(u,u')>R$ whenever $u\in U$ and $u'\in U'$. A family of subsets $\U{}$ of a metric space $X$ will be said to be \df{$R$-disjoint} if $U$ and $U'$ are $R$-disjoint whenever $U$ and $U'$ are distinct elements of $\U{}$. 
We define the \df{mesh} of a collection $\U{}$ of subsets of $X$ by $\mesh(\U{})=\sup\{\diam(U)\colon U\in \U{}\}$. A family of subsets $\U{}$ of a metric space $X$ will be said to be \df{uniformly bounded} if $\mesh(\U{})<\infty$. We will occasionally describe a family as being \df{$B$-bounded} when $\mesh(\U{})\le B$ for some $B>0$. 

Asymptotic dimension was defined by Gromov \cite{gromov93} as an asymptotic analog of covering dimension.

\begin{definition} A metric space $X$ is said to have \df{asymptotic dimension at most $n$}, $\as X\le n$, if for every $R>0$ there are $(n+1)$-many uniformly bounded families of $R$-disjoint sets $\U{0},\U{1},\ldots, \U{n}$ so that $\cup_i\U{i}$ covers $X$. 
\end{definition}

Occasionally we will need more control over the disjointness and mesh of covers of a collection of sets. We say that a family $\{X_\alpha\}_\alpha$ of subsets of a metric space $X$ has asymptotic dimension no more than $n$ \df{uniformly} if for every $R$ there is a $B$ so that each $X_\alpha$ can be covered by at most $(n+1)$-many $R$-disjoint families of $B$-bounded sets. 

Dranishnikov defined asymptotic property~C \cite{dranish2000} as an asymptotic analog of Haver's topological property~C~\cite{haver}. 

\begin{definition}
  A metric space $X$ has \df{asymptotic property~C} if
  for each (infinite) sequence $R_1, R_2, \ldots$ of real numbers
  there exists a finite sequence $\U1, \U2, \ldots, \U{n}$ of families of
  subsets of $X$ such that
  \begin{enumerate}
  \item $\bigcup_i \U{i}$ covers $X$,
  \item $\U{i}$ is $R_i$-disjoint,
  \item $\U{i}$ is uniformly bounded.
  \end{enumerate}
\end{definition}

There is no loss of generality in assuming that the numbers $R_i$ comprise a non-decreasing sequence. It is clear that spaces with finite asymptotic dimension have asymptotic property~C; it is also true that metric spaces with asymptotic property~C have Yu's property A~\cite[Theorem 7.11]{dranish2000}. 


It is easy to see that the union of a finite collection of sets with asymptotic property~C has asymptotic property~C. In fact, this is true for certain types of infinite unions, too (see~\cite[Theorem 4.2]{bell-moran2015}).

\begin{proposition}\label{prop:apc-union}
Let $X$ and $Y$ have asymptotic property~C. Then $X\cup Y$ has asymptotic property~C.
\end{proposition}

Finally, we mention that the concepts of asymptotic dimension and asymptotic property~C (among others) belong to the realm of so-called coarse geometry. It is possible to define these concepts in the context of the coarse category without reference to a metric, (see \cite{bell-moran-nagorko2016,grave2006,roe2003}); however, we do not pursue this course further in this work.

\begin{definition}
Let $f \colon X\to Y$ be a map of metric spaces. We say that $f$ is

\begin{enumerate}[(i)]
\item \df{effectively proper} if there is some proper (see Definition~\ref{def:proper}), non-decreasing $\rho_1:[0,\infty)\to[0,\infty)$ such that $\rho_1(\dist_X(x,x'))\le \dist_Y(f(x),f(x'))$;
\item \df{uniformly expansive} if there is a non-decreasing $\rho_2:[0,\infty)\to [0,\infty)$ such that $\dist_Y(f(x),f(x'))\le\rho_2(\dist_X(x,x'))$;
\item \df{coarsely uniform embedding} if it is both effectively proper and uniformly expansive;
\item \df{coarsely surjective} if there is some $R$ so that for each $y\in Y$, there is some $x\in X$ for which $\dist_Y(y,f(x))<R$;
\item a \df{coarse equivalence} if it is a coarsely uniform embedding that is coarsely surjective.
\end{enumerate}
\end{definition}

It is straightforward to verify that asymptotic dimension and asymptotic property~C are invariants of coarse equivalence.

\begin{definition}\label{def:proper}
A metric space $(X,\dist)$ is said to be \df{proper} if closed balls are compact. We also say that the metric itself is proper in this case. A map $f:X\to Y$ between topological spaces is called \df{proper} if inverse images of compact sets are compact. 

A metric $d$ on a group $\Gamma$ will be called \df{left-invariant} if $\dist(gh,gh')=\dist(h,h')$ for all $g,\,h,$ and $h'$ in $\Gamma$.
\end{definition}

Let $\Gamma$ be a finitely generated group. Fix a finite generating set $S$ that is symmetric in the sense that $s\in S$ implies $s^{-1}\in S$. The pair $(\Gamma, S)$ carries a natural proper left-invariant metric called the word metric. This metric is given by first defining a norm on $\Gamma$ as follows: $\|g\|_S=0$ if and only if $g$ is the identity element; otherwise $\|g\|_S$ is the length of the shortest word in the alphabet $S$ that presents the element $g$. The (left-invariant) \df{word metric} is given by $\dist_S(g,h)=\|g^{-1}h\|_S$. It is easy to see that any two word metrics corresponding to different (finite) generating sets give rise to coarsely equivalent spaces. In the case of a general countable group $\Gamma$, which may not be finitely generated, Dranishnikov and Smith showed that up to coarse equivalence there is a unique left-invariant proper metric on $\Gamma$~\cite[Proposition 1.1]{dranishnikov-smith-2006}. Moreover, each such metric arises from a weight function as follows.

\begin{definition}
Let $\Gamma$ be a countable group with (possibly infinite) generating set $S$. Give $S$ the discrete topology. A \df{weight function} on $S$ is a map $w\colon S\to [0,\infty)$ that is positive, proper, and has the property that $w(s)=w(s^{-1})$. 
\end{definition}

\begin{theorem}\cite[Proposition 1.3]{dranishnikov-smith-2006}
Every weight function $w$ on a countable group defines a proper norm (which induces a proper left-invariant metric). The norm is given by 
\[\|\gamma\|_w=\inf\left\{\sum_{i=1}^nw(s_i)\colon \gamma=s_1\cdots s_n,\, s_i\in S\right\}\] and the corresponding metric is $\dist_w(g,h)=\|g^{-1}h\|_w$.
\end{theorem}

\begin{definition}\label{def:R-stab}
An \df{action} of a group $\Gamma$ on a topological space $X$ is a homomorphism $\varphi:\Gamma\to Homeo(X)$. We denote $\varphi(\gamma)(x)$ by $\gamma.x$. The \df{orbit} of $x_0\in X$ is the equivalence class of all $\{x\in X\colon \exists_{g\in \Gamma}\, x=g.x_0\}$. The action is \df{transitive} if there is only one orbit. For $x\in X$, the \df{stabilizer} of $x$ is denoted $\Gamma_x$ and is the subgroup of $\Gamma$ defined by $\Gamma_x=\{\gamma\in \Gamma\colon \gamma.x=x\}$. Let $\Gamma$ be a group acting on a metric space $X$. The action of $\Gamma$ on $X$ is said to be an \df{action by isometries} (or an \df{isometric action}) if $\varphi(\Gamma)\subset Isom(X)$. 
For $R>0$ define the \df{$R$-stabilizer of $x_0\in X$} by $W_R(x_0)=\{\gamma\in\Gamma\colon \dist_X(\gamma.x_0,x_0)\}\le R$.
\end{definition}


\section{Direct products preserve asymptotic property~C}

Several authors have studied the so-called permanence properties -- that is, the extent to which properties are preserved by unions, products, etc -- of coarse invariants such as finite asymptotic dimension~\cite{bell-dranishnikov2001}, finite decomposition complexity~\cite{guentner2014,dydak-virk2016}, or property A~\cite{bell2003}. The permanence properties of asymptotic property~C have proved to be slightly more elusive than others, with incremental special cases being proved for unions ~\cite{bell-moran2015} and certain types of group extensions and products~\cite{beckhardt2015,beckhardt-goldfarb2016}. Indeed, although it had been conjectured for some time, only recently was a group with infinite asymptotic dimension and asymptotic property~C shown to exist~\cite{yamauchi2015}.

Let $(X,\dist_X)$ and $(Y,\dist_Y)$ be metric spaces. We can define a metric on the product $X\times Y$ by $\dist_{X\times Y}(x\times y, x'\times y')=\sqrt{\left(\dist_X(x,x')\right)^2+\left(\dist_Y(y,y')\right)^2}$.

The main goal of this section is to answer the following question from the Lviv Topological Seminar: 
\cite[Question 1.3, p.~656]{openproblems} Let $X$ and $Y$ be two metric spaces with asymptotic property~C. Does $X \times Y$ have asymptotic property~C?

\begin{theorem}\label{thm:direct-product}
Let $X$ and $Y$ be metric spaces with asymptotic property~C. Then, $X\times Y$ has asymptotic property~C.
\end{theorem}

The idea for the proof was inspired by the proof that a product of a compact space with (topological) property~C and a space with property~C has property~C~\cite[Theorem 3]{rohm1990}.

\begin{proof}
Let $R_1\le R_2\le \cdots$ be a given non-decreasing sequence of positive numbers. Rearrange the sequence into subsequences $R_{i,1}\le R_{i,2}\le\cdots$ using the rearrangement shown in Figure~\ref{fig:rearrangement}.


\begin{figure}
\begin{center}
\includegraphics[width=.9\textwidth]{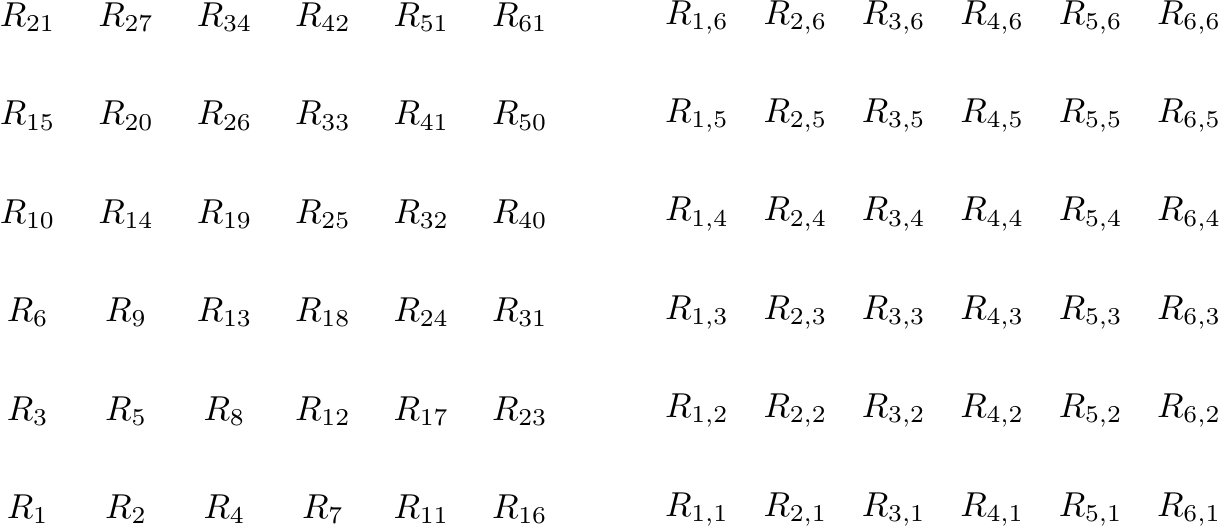}
\end{center}
\caption{We arrange the given sequence as indicated and assign new labels corresponding to the row and column in which the number now appears.}\label{fig:rearrangement}
\end{figure}

\begin{figure}
\includegraphics[width=.9\textwidth]{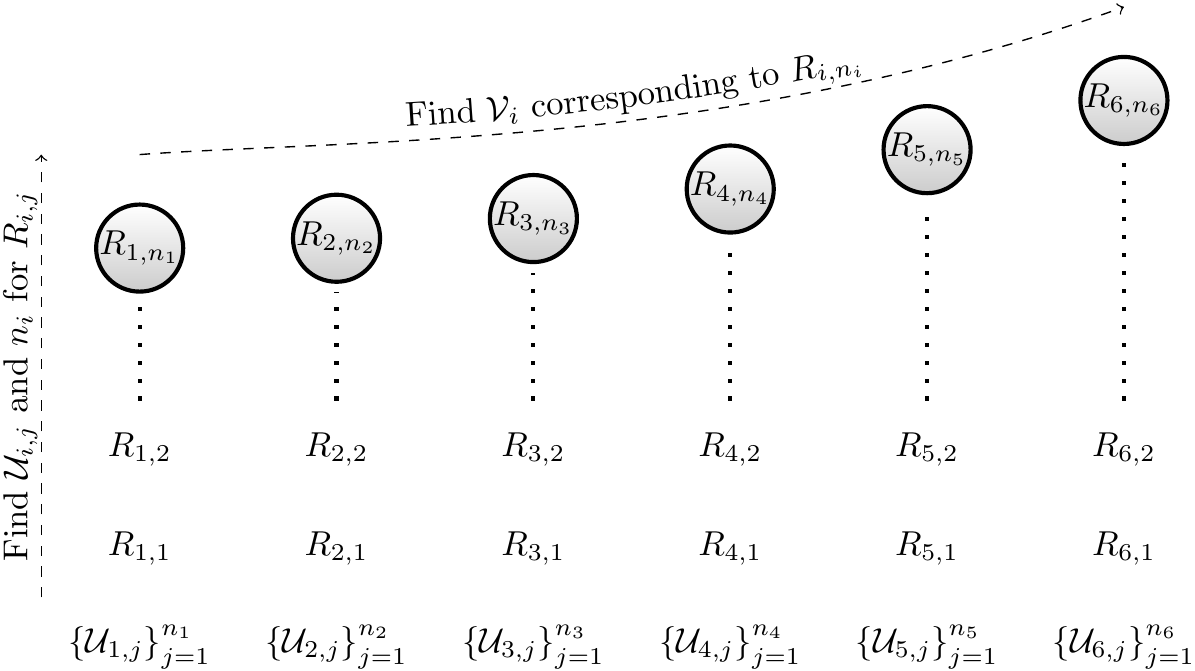}
\caption{We find covers $\{\U{i,j}\}_{j=1}^{n_i}$ of $X$ for each column $R_{i,1},R_{i,2},\ldots,$ and then construct a cover $\{\V{i}\}_{i=1}^m$ of $Y$ corresponding to $R_{i,n_i}$}\label{fig:U and V}
\end{figure}

Fix $i=1$ so that we are considering the column $R_{1,j}$ with $j=1,2,\ldots$. Let $\U{1,1}, \U{1,2},\ldots, \U{1,n_1}$ be uniformly bounded families of subsets of $X$ so that  $\U{1,j}$ is $R_{1,j}$-disjoint and the union $\bigcup_j\U{1,j}$ covers $X$. Similarly, for each $i=2,3,\ldots$, we construct finite collections of $R_{i,j}$-disjoint families $\U{i,j}$ ($j=1,2,\ldots,n_i$) of uniformly bounded subsets of $X$ whose union covers $X$. This construction yields a sequence of positive numbers $R_{i,n_i}$ with the property that $R_{i,n_i}\ge R_{i,j}$ for all $j\le n_i$. Using this sequence we obtain covers $\V{i}$ ($i=1,2,\ldots,m$) of $Y$, see Figure~\ref{fig:U and V}.

Let $\W{i,j}=\{U\times V\colon U\in \U{i,j}, V\in\V{i}\}$.

Given $(x,y)\in X\times Y$ there is some $i_0$ so that $\V{i_0}$ contains an element $V$ containing $y$ since the union of the $\V{i}$ covers $Y$. Then, there is some $j_0$ so that $\U{i_0,j_0}$ contains a set $U$ containing $x$, since each $\bigcup_j\U{i,j}$ is a cover of $X$ for each $i$. Thus, $U\times V$ is some $W\in\W{i_0,j_0}$ that contains $(x,y)$. Thus, the collection $\bigcup_{i,j}\W{i,j}$ covers $X\times Y$. 

The family $\W{i,j}$ is uniformly bounded; indeed, if $W\in\W{i,j}$ then $\diam(W)^2\le \mesh(\U{i,j})^2+\mesh(\V{i})^2$.

Next, we check that each family $\W{i,j}$ is $R_{i,j}$-disjoint; we suppose that distinct elements $W$ and $W'$ are in some fixed $\W{i,j}$. Write $W=U\times V$ and $W'=U'\times V'$. Take arbitrary elements $(u,v)\in W$ and $(u',v')\in W'$. If $U=U'$ then $\dist((u,v),(u',v'))\ge \dist_Y(v,v')> R_{i,n_i}$. If $V=V'$, then $\dist((u,v),(u',v'))\ge\dist_X(u,u')> R_{i,j}$. Otherwise, $\dist((u,v),(u',v'))\ge \min\{\dist_X(u,u'),\dist_Y(v,v')\} > \min\{R_{i,n_i},R_{i,j}\}=R_{i,j}$ since $R_{i,j}\le R_{i,n_i}$ for each $i$ and for each $j\le n_i$. Thus, $\W{i,j}$ is $R_{i,j}$-disjoint.

Finally, we rearrange the $\W{i,j}$ back into a single sequence as in Figure~\ref{fig:rearrangement} to obtain $\W{k_1}, \W{k_2},\ldots, \W{k_N}$ (where some indices may have been omitted). Notice that this produces a finite sequence $\{k_1,k_2,\ldots,k_N\}$ with the property that $k_i\ge i$ for each $i$. Thus, $R_{k_i}\ge R_i$ for each $i$. Thus, the collection $\W{k_i}$ is $R_{k_i}$-disjoint and therefore it is $R_i$ disjoint. Therefore, we have a finite collection $\W{k_1},\W{k_2},\ldots, \W{k_N}$ of uniformly bounded $R_i$-disjoint families of subsets of $X\times Y$ with the property that $\bigcup\W{k_i}$ covers $X\times Y$. 
\end{proof}


\section{The Fibering Lemma}

In this section we prove a Hurewicz-type theorem for asymptotic property~C that is motivated by the technique of the proof of Theorem~\ref{thm:direct-product}. The classical Hurewicz theorem on mappings that lower dimension goes back to 1927 and can be stated as follows:

\begin{hurewicz theorem}\cite[Theorem 1.12.4]{engelking1995}
If $f\colon X\to Y$ is a closed mapping of separable metric spaces and there is an integer $k\ge 0$ such that $\dim f^{-1}(y)\le k$ for every $y\in Y$, then $\dim X\le \dim Y+k$.
\end{hurewicz theorem}

This theorem was first translated to asymptotic dimension for finitely generated groups in~\cite{b-d-hurewicz}. It was proved for countable groups in~\cite{dranishnikov-smith-2006} and translated to Assouad-Nagata dimension in~\cite{BDLM}. A version of the theorem for metric spaces appears in~\cite{dydak-virk2016}. In a more general context, this theorem belongs to the class of fibering theorems, (see, for example~\cite{guentner2014}) in which a property is preserved by a mapping when the range and fibers have this property.

\subsection{Fibers with uniform asymptotic property~C}

\begin{definition}\label{def:coarse fibers}
  Let $f \colon X \to Y$ be a map of metric spaces.
  Let $M$ be a real number.   
  We say that $A \subset X$ is a \df{coarse fiber at scale $M$} 
    if $\diam f(A) < M$.
\end{definition}

\begin{definition}\label{def:apc fibers}
  Let $f \colon X \to Y$ be a map of metric spaces.
  We say that \df{coarse fibers of $f$ have uniform asymptotic property~C} if 
  for each infinite sequence $R_1, R_2, R_3, \ldots$ of real numbers
  there exists an integer $k$ such that 
  for each real number~$M$ 
  there exists real number~$B$ such that 
  for every $A \subset X$ with $\diam f(A) < M$ 
  there exists a sequence $\U{1}, \U{2}, \ldots, \U{k}$ of families of subsets of $A$ that satisfies the following conditions
  \begin{enumerate}
  \item $\bigcup_{i = 1}^k \U{i}$ covers $A$,
  \item for each $i$, $\U{i}$ is $R_i$-disjoint,
  \item for each $i$, $\mesh \U{i} \leq B$.
  \end{enumerate}
\end{definition}

\begin{comm}\label{comm:apc fibers}
  Note the order of quantifiers in Definition~\ref{def:apc fibers}: 
    the number $k$ of families depends only on sequence $R_1, R_2, \ldots$; 
    the bound~$B$ on meshes of the families depends on the sequence $R_1, R_2, \ldots$, the number $k$, and the scale $M$ of fibers.
\end{comm}

If $f:X\times Y\to Y$ is the projection map from the product of metric spaces $X\times Y$, then it is straightforward to show that coarse fibers of $f$ have uniform asymptotic property C if $X$ has asymptotic property C. 

\begin{fibering lemma}
  If $f \colon X \to Y$ is 
  a uniformly expansive map of metric spaces, 
  $Y$ has asymptotic property~C, and 
  coarse fibers of~$f$ have uniform asymptotic property~C,
  then $X$ has asymptotic property~C.
\end{fibering lemma}

It is known that topological property~C is preserved by pre-images if the map is closed, fibers have topological property~C and the codomain is paracompact, see~\cite{hattori-yamada1989}.

\begin{proof}
  Let $\rho \colon [0,\infty) \to [0,\infty)$ be strictly increasing
    such that $\dist_Y(f(x),f(x'))\le\rho(\dist_X(x,x'))$.
  Let $R_1, R_2, R_3, \ldots$ be an infinite sequence of real numbers.
  Without any loss of generality we may assume that $R_1, R_2, R_3, \ldots$
    is increasing.
  
  Rearrange the sequence $R_1 \leq R_2 \leq R_3 \leq \ldots$ into subsequences $R_{i,1} \leq R_{i,2} \leq R_{i,3} \leq \ldots$ using the rearrangement shown in Figure~\ref{fig:rearrangement}.
  For each $i$ let $n_i$ be the number of families obtained by applying Definition~\ref{def:apc fibers} to the map $f$ and the sequence $R_{i,1} \leq R_{i,2} \leq R_{i,3} \leq \cdots$.
  Consider the sequence
  \[
    \rho(R_{1, n_1}), \rho(R_{2, n_2}), \rho(R_{3, n_3}), \ldots
  \]
  Using the assumption that $Y$ has asymptotic property~C we obtain a sequence
  $\V{1}, \V{2},\ldots, \V{m}$ of families of subsets of $Y$ such that
  \begin{enumerate}
  \item $\bigcup_{i=1}^m \V{i}$ covers $Y$.
  \item for each $i$, $\V{i}$ is $\rho(R_{i, n_i})$-disjoint.
  \item for each $i$, $\V{i}$ is uniformly bounded.
  \end{enumerate}
  Let $M_i = \mesh \V{i}$. 
  Apply Definition~\ref{def:apc fibers} to the map $f$, the sequence $R_{i,1} \leq R_{i, 2} \leq \cdots \leq R_{i, n_i}$, and the scale $M_i$ to obtain a bound $B_i$.
  
  Let $V \in \V{i}$. The set $f^{-1}(V)$ is a coarse fiber of $f$ at scale $M_i$. By Definition~\ref{def:apc fibers} applied to the sequence $R_{i,1} \leq R_{i, 2} \leq \cdots \leq R_{i, n_i}$ and the coarse fiber $f^{-1}(V)$ at scale $M_i$, there exists a sequence $\mathcal{U}_1^V, \mathcal{U}_1^V, \ldots, \mathcal{U}_{n_i}^V$ of families of subsets of $f^{-1}(V)$ such that
  \begin{enumerate}
  \item $\bigcup_{j=1}^{n_i} \mathcal{U}_j^V$ covers $f^{-1}(V)$.
  \item for each $j$, $\mathcal{U}_j^V$ is $R_{i, j}$-disjoint.
  \item for each $j$, $\mesh \mathcal{U}_j^V \leq B_i$.
  \end{enumerate}

Observe that if $V, W \in \V{i}$, if $v\in V$, $w\in W$, and $V \neq W$, then $\dist_X (f^{-1}(v), f^{-1}(w)) > R_{i, n_i}$. Since $R_{i, n_i} \geq R_{i, j}$ for $j \leq n_i$, the family
\[
\U{i,j} = \bigcup_{V \in \V{i}} \mathcal{U}^V_j
\]
is $R_{i, j}$-disjoint and $\mesh \U{i, j} < B_i$.
Hence we obtain a collection of uniformly bounded families $\{ \U{i,j} \colon i \leq m, j \leq n_i \}$ that are $R_{i,j}$-disjoint and whose union covers $X$.
Rearranging these families back into the sequence $R_1, R_2, \ldots$ (possibly leaving some places empty) shows that $X$ has asymptotic property~C.
\end{proof}

By the comments following Comment~\ref{comm:apc fibers}, we see that the product theorem (Theorem~\ref{thm:direct-product}) can be deduced from the Fibering Lemma.

\subsection{Fibers with uniformly bounded asymptotic dimension}

We turn our attention to the preservation of coarse properties discussed in \cite{dydak-virk2016}, from which we recall the following definition.

\begin{definition} Given a function $f:X\to Y$ of metric spaces, its \df{asymptotic dimension} $\as(f)$ is the supremum of $\as(A)$ such that $A\subset X$ and $\as(f(A))=0$.
\end{definition}

\begin{lemma}\label{lem:asdim-kernel-spaces}
Let $f\colon X\to Y$ be a uniformly expansive map of metric spaces. The following are equivalent.
\begin{enumerate}
	\item  For every $M>0$ and for every $R>0$ there is a $B>0$ so that if $U\subset X$ and $\diam(f(U))<M$, then $U$ can be covered by $(n+1)$-many $R$-disjoint families of $B$-bounded subsets of $X$.
	\item $\as f \le n$.
\end{enumerate}
\end{lemma}

\begin{proof}
(1) $\implies$ (2): Since $f$ is uniformly expansive there is an increasing control function $\rho:[0,\infty)\to [0,\infty)$ so that $\dist_Y(f(x),f(x'))\le\rho(\dist_X(x,x'))$. Let $R>0$ be given and take an arbitrary $A\subset X$ with the property that $\as f(A)=0$. We may cover $f(A)$ with uniformly bounded sets $U_i$ that are $\rho^{-1} (R)$-disjoint. Let $M$ be an upper bound for $\diam U_i$ and take $B>0$ as above. Then, each $f^{-1}(U_i)$ can be covered by the union of $n+1$ families $\U{i,j}$ of $R$-disjoint, uniformly $B$-bounded subsets of $f^{-1}(U_i)$. By the choice of the $U_i$, we see that if $U_i\neq U_j$ then $\dist(f^{-1}(u),f^{-1}(v))> R$ whenever $u\in U_i$ and $v\in U_j$. Thus, by taking the collection $\V{j}=\U{i,j}$ for $j=0,1,\ldots,n$ we obtain a collection of $(n+1)$-many $R$-disjoint families of uniformly bounded sets whose union covers $A$. Thus, $\as A\le n$ and since $A$ was arbitrary, we conclude that $\as f\le n$.    

(2) $\implies$ (1): Suppose that $\as f\le n.$ Fix $R>0$ and $M>0$. For each $U\subset X$ with $\diam(f(U))< M$, let $B(U)$ denote the smallest integer $B$ for which $U$ can be expressed as a union of $B$-bounded, $R$-disjoint families $\U{0},\ldots,\U{n}$.

If $\sup\{B(U)\colon \diam(f(U))< M\}<\infty$ then there is nothing to prove. Thus, we assume that we can find a sequence of subsets $U_1,U_2,\ldots$ of $X$ with $\diam(f(U_i))<M$ for which $B(U_1), B(U_2),\ldots$ is unbounded. We may take the sequence of such $B(U_i)$ to be monotonically increasing. If the set $\{f(U_1), f(U_2), \ldots\}$ is bounded, then $\as\left(\bigcup_i f(U_i)\right)=0$ and therefore the sequence of the $B(U_i)$ is bounded from above. Otherwise, the sequence $f(U_i)$ is unbounded. Thus, we take $i_1=1$. Because the $f(U_i)$ are bounded, yet the sequence $\{f(U_i)\}$ is unbounded, for each $k>1$, we can find an increasing sequence of indices $i_k$ so that $\dist(\bigcup_{j=1}^{k-1}f(U_{i_j}),f(U_k))>2^k$. (Here, the distance between subsets is defined by $\dist(A,B)=\inf\{\dist(a,b)\colon a\in A, b\in B\}$.) Thus, $\bigcup_{k=1}^\infty f(U_{i_j})$ has asymptotic dimension $0$, which contradicts the fact that the $B(U_i)$ increase monotonically to infinity since $\as f\le n$.
\end{proof}

Dydak and Virk have the following theorem regarding such maps and property~C:

\begin{theorem}[Theorem 6.3 \cite{dydak-virk2016}]
Suppose that $f:X\to Y$ is coarsely surjective and $\as(f)=0$. Then, if $Y$ has asymptotic property~C, $X$ has asymptotic property~C.
\end{theorem}

This naturally leads them to the following question, which we are able to settle in the affirmative~\cite[Question 6.4]{dydak-virk2016}: 
suppose $f:X\to Y$ is uniformly expansive, coarsely surjective, and of finite asymptotic dimension $\as(f)$. If $Y$ has asymptotic property~C, does $X$ have asymptotic property~C?

\begin{corollary} \label{cor:as(f)} Suppose that $f:X\to Y$ is a uniformly expansive coarsely surjective map of finite asymptotic dimension $\as(f)$ and that $Y$ has asymptotic property~C. Then, $X$ has asymptotic property~C. 
\end{corollary}

\begin{proof}
Suppose that $\as(f)=n$. Then, given a sequence $R_1\le R_2\le\cdots$, we take $k=n+1$ and apply Lemma~\ref{lem:asdim-kernel-spaces} to $R_{n+1}$ to find that for any $M>0$ there is a $B$ so that whenever $\diam f(A)<M$, we can cover $A$ by $(n+1)$-many $R_{n+1}$-disjoint families $\U{1},\ldots, \U{n+1}$ that are each $B$-bounded. Since the sequence $R_i$ is non-decreasing, we see that this shows that the coarse fibers of $f$ have uniform asymptotic property~C. Thus, by the Fibering Lemma, we conclude that $X$ has asymptotic property~C.
\end{proof}


\section{Fibers of group homomorphisms}

In the assumptions of the Fibering Lemma we impose the asymptotic property~C condition \emph{uniformly} on \emph{coarse} fibers. 
The following example shows that it is indeed necessary.
In the present section we show that for group homomorphisms in some cases the uniformity of the fiber condition is implicit.

\begin{example}
  Let $\{ 0, 1 \}^n$ be the set of vertices of an $n$-dimensional cube endowed with the $\ell_1$-metric.
  The disjoint union $\coprod_{n=1}^\infty \{ 0, 1 \}^n$ can be metrized in such a way that it is a locally finite metric space that fails to have property A~\cite{nowak2007} and hence fails to have asymptotic property~C as well~\cite[Theorem 7.11]{dranish2000}.
  Consider the map $\coprod_{n=1}^\infty \{ 0, 1 \}^n \to \mathbb{Z}$ that collapses $\{ 0, 1 \}^n$ into a single integer $n^2$.
  It is a uniformly expanding map onto a zero-dimensional space.
  Its coarse fibers are all bounded (hence they are zero-dimensional and have asymptotic property~C).
\end{example}

\subsection{Fibers vs coarse fibers}

Throughout this section, $G$ and $H$ denote countable groups endowed with proper left-invariant metrics.

%
%
%
\begin{lemma}\label{lem:group-actions-W_R}
Let $G$ be a countable group with a proper left-invariant metric acting by isometries on a metric space $X$. Let $x_0\in X$ and define $\varphi:G\to X$ by $\varphi(g)=g.x_0$. Fix $n$ and suppose that for every $R>0$, the $R$-stabilizer (see Definition \ref{def:R-stab}) satisfies $\as W_R(x_0)\le n$. Then coarse fibers of $\varphi$ have uniform asymptotic property~C. 
\end{lemma}

We remark that the $R$-stabilizer is a coarse fiber in the sense of Definition~\ref{def:coarse fibers}.

\begin{proof}
We may assume that the action is transitive. Suppose that a sequence $R_1\le R_2\le\cdots$ is given and take $k=n+1$. Now, if $M>0$ is given we may use the fact that $\as W_M(x_0)\le k-1$ to find families of $R_{k}$-disjoint sets $\U{1},\U{2},\ldots,\U{k}$ so that $\cup_i\U{i}$ covers $W_M(x_0)$ and $\mesh(\U{i})\le B$ for some $B$.

Let $A\subset G$ have the property that $\diam \left(\varphi(A)\right)\le M$. Then, take any $g_A\in G$ and observe that $g_A(x_0)\in \varphi(A)$. Thus $\varphi(A)\subset B_M(g_A.x_0)$ and so $g_A^{-1} A\subset W_M(x_0)$. Since the metric on $G$ is left-invariant, the families $\V{i}$ given by $\{g_A^{-1}U\colon U\in\U{i}\}$ have the property that they are $R_k$-disjoint and $B$-bounded. Moreover, their union covers $A$ as required.
\end{proof}

We actually proved a stronger statement: for each $R$ the coarse fibers at scale $R$ have asymptotic dimension at most $n$ uniformly.

\begin{lemma}\label{lem:group-action-BG-space}
Let $\varphi:G\to X$ describe the transitive action by isometries of a countable discrete group $G$ on a proper discrete metric space $X$. If the stabilizer $G_x$ has finite asymptotic dimension for some $x\in X$, then coarse fibers of the map $\varphi:G\to X$ describing the action have uniform asymptotic property~C.
\end{lemma}

\begin{proof}
Let $R>0$ be given and consider $B_R(x)\subset X$. Since this set is finite, it follows that $W_R(x)$ is a finite union of sets that are isometric to $G_x$. Hence, this union has the same asymptotic dimension as $G_x$ and the result follows from Lemma~\ref{lem:group-actions-W_R}.
\end{proof}

\begin{lemma}\label{lem:G -> H coarse fibers}
  If the kernel of a group homomorphism $f \colon G \to H$ has finite asymptotic dimension, then coarse fibers of $f$ have uniform asymptotic property~C. 
\end{lemma}
We do not know whether the weaker assumption that the kernel of $f$ has asymptotic property~C is sufficient, see Question~\ref{q:apc fiber}.

\begin{proof} Define an action of $G$ on $H$ by $g.h\mapsto f(g)h$. Notice that this action is isometric; i.e., $\dist(g.h,g.h')=\dist(h,h')$ for all $g\in G$. Also, the stabilizer of this action at $e$ is $\ker f$. Since $H$ is assumed to be a countable group with a proper metric, as a metric space it has bounded geometry and the result follows from Lemma~\ref{lem:group-action-BG-space}.
\end{proof}

\begin{lemma}
  If the kernel of a projection $p \colon G \times H \to G$ has asymptotic property~C, then coarse fibers of $p$ have uniform asymptotic property~C.
\end{lemma}
\begin{proof} We observe that the kernel of $p$ is $\{e\}\times H$, which is isometric to $H$. Let $R_1\le R_2\le\cdots $ be a given sequence. Take families $\U{1},\U{2},\ldots,\U{k}$ of subsets of $H$ so that $\U{i}$ is $R_i$-disjoint, so that $\mesh\U{i}$ is finite for each $i$, and so that $\cup_i\U{i}$ covers $H$. Given $M>0$ we take $B=M+\max_i\{\mesh\U{i}\}$. If $A\subset G\times H$ has the property that $\diam(p(A))\le M$, then there is some $g\in G$ so that $p(A)\subset B^G(g, M)$, where the superscript indicates that the $M$-ball is taken in the group $G$. Consider the collection $\V{i}$ of subsets of $G\times H$ given by $\{B^G(g, M)\times U\colon U\in\U{i}\}$. Observe that this collection covers $A$, $\mesh\V{i}\le \mesh\U{i}+M\le B$, and each $\V{i}$ is $R_i$-disjoint.
\end{proof}

We would like to remark that we did not require the bounded geometry condition on $G$ or $H$ in the previous proof. Indeed, the same technique can be used to show that if point fibers of a projection $p\colon X\times Y\to X$ have asymptotic property~C, then coarse fibers of $p$ have uniform asymptotic property~C.




In \cite{beckhardt2015} it is shown that a finitely generated group acting on a space with asymptotic property~C will have asymptotic property~C provided the $R$-stabilizers have finite asymptotic dimension. We can use our fibering theorem to recover this result, which we state for arbitrary countable groups with proper left-invariant metrics.

\begin{theorem}\label{thm:group-actions}
Let $G$ be a countable group with a proper left-invariant metric acting by isometries on a metric space $X$. Let $x_0\in X$ and define $\phi:G\to X$ by $\phi(g)=g.x_0$. If for every $R>0$, $\as W_R(x_0)\le n$ and $X$ has asymptotic property~C, then $G$ has asymptotic property~C. 
\end{theorem}

\begin{proof}
This will follow from the Fibering Lemma and Lemma \ref{lem:group-actions-W_R} provided we can show that $\phi$ is uniformly expansive.

To this end, by \cite{dranishnikov-smith-2006} we know that the metric on $G$ is a weighted word metric. Therefore, we assume that $S=\{s_i\}$ is a generating set for $G$ with corresponding weights $w(s_i)$. By their definition and the fact that the metric is proper, the weights have the following property. For each fixed $N\in \mathbb{N}$, there are only finitely many (finite) sequences of weights whose sums are bounded from above by $N$; that is, the collection of all sequences $\left(w(s_{i_1},w(s_{i_2}),\ldots \right)$ for which $\sum w(s_{i_j})\le N$ is a finite collection. Thus, it makes sense to define a function $\rho:\mathbb{N}\to \mathbb{N}$ by $\rho(N)=\max\{\sum \dist(s_{i_j}.x_0,x_0)\}$ where the maximum is taken over all sequences of $s_i$ for which $\sum w(s_{i_j})\le N$.

Then, by applying the triangle inequality and the fact that the metric is left-invariant, we find $\dist_X(g'.x_0,g.x_0)=\dist_X(g^{-1}g'.x_0,x_0)\le \sum_{j=1}^k \dist_X(s_{i_j}.x_0,x_0)$, where $g^{-1}g'=s_{i_1}s_{i_2}\cdots s_{i_k}$. But, now $\sum_{j=1}^k \dist_X(s_{i_j}.x_0,x_0)\le \rho(\|g^{-1}g'\|)=\rho(\dist_G(g',g))$ and so we conclude that $\dist_X(\phi(g'),\phi(g))\le\rho(\dist_G(g',g))$. \end{proof}

%

\begin{theorem}\label{thm:surjection}
Let $\phi:G\to H$ be a surjective homomorphism of finitely generated groups; suppose that $H$ has property $C$. If $\as\ker\phi<\infty$
then $G$ has asymptotic property~C.
\end{theorem}

\begin{proof} 
Let $S$ be a generating set for $H$ with corresponding weights $W$. For each $s\in S$, let $\tilde{s}\in \phi^{-1}(s)$ be some element of $G$. Extend the set $\{\tilde{s}\colon s\in S\}$ to a generating set $T$ for $G$ with the property that the weight of each $\tilde{s}$ is the same as the weight of the corresponding $s$. This yields a left-invariant proper metric on $G$, which is coarsely equivalent to any other choice of left-invariant proper metric on $G$, so it suffices to show that $G$ has asymptotic property~C in this metric. In this metric, $\phi$ is coarsely expansive and so the result follows from the Fibering Lemma and Lemma~\ref{lem:G -> H coarse fibers}.
%
%
%
\end{proof}

\begin{corollary}\ 
If $G$ and $H$ have asymptotic property~C then $G\ast H$ has asymptotic property~C.
\end{corollary}

\begin{proof}
We consider the exact sequence 
\[1\to [[G,H]]\to G\ast H\to G\times H\to 1\hbox{.}\]
The commutator group $[[G,H]]$ is free~\cite{serre} and so as a subspace of $G\ast H$ it is quasi-isometric to a tree (albeit one in which edges can have arbitrarily long length). Thus, its asymptotic dimension is $1$. By Theorem~\ref{thm:direct-product}, $G\times H$ has asymptotic property C. We apply Theorem~\ref{thm:surjection} to complete the proof.
\end{proof}


\section{Free products preserve asymptotic property~C}

In the present section we use the Decomposition Lemma to prove that the free product of metric spaces preserves asymptotic property~C; however, the Decomposition Lemma has interesting generalizations that make it interesting in its own right.

\subsection{The Decomposition Lemma}

\begin{lemma}\label{lem:decomposition}
  Let $X$ be a metric space. 
  Assume that there exists $k$ such that for each (infinite) sequence $R_1, R_2, \ldots$ of real numbers there exists a finite sequence $\U1, \U2, \ldots, \U{n}$ of families of subsets of $X$ such that
  \begin{enumerate}
  \item $\bigcup_i \U{i}$ covers $X$,
  \item $\U{i}$ is $R_i$-disjoint,
  \item $\U{i}$ has asymptotic dimension bounded by $k-1$ uniformly.
  \end{enumerate}
  Then $X$ has asymptotic property~C.
\end{lemma}
\begin{proof} Let $R_1\le R_2\le\cdots$ be a given sequence of real numbers. For $i\in\{1,2,\ldots, n\}$, take $\U{i}$ to be an $R_{ik}$-disjoint collection of subsets of $X$ that has asymptotic dimension less than $k$ uniformly so that $\bigcup_i\U{i}$ covers $X$. 

Fix some $i$ and suppose that $U\in\U{i}$. Using the uniformity of the asymptotic dimension of the $\U{i}$, take some $B$ and take families $\V{j}^U$ ($j=1,\ldots, k$) that comprise an $R_{ik}$-disjoint collection of subsets of $U$ whose diameter is bounded by $B$ so that the union covers $U$. Let  $\V{(i-1)k+j}$ denote the collection $\{V\in \V{j}^U\mid U\in\U{i}\}.$ Since the family $\U{i}$ is itself $R_{ik}$-disjoint, it is clear that the collection $\V{(i-1)k+j}$ ($j=1,\ldots, k$) will be an $R_{ik}$-disjoint, $B$-bounded cover of $\cup_{U\in\U{i}} U$. Since we assume the given sequence to be non-decreasing, we have found a uniformly bounded cover of $X$ by $nk$ families $\V{1},\ldots,\V{nk}$ of subsets of $X$ 
with the property that $\V{t}$ is $R_{t}$-bounded, for each $t=1,\ldots,nk$. Thus, $X$ has property~C, as required. 
\end{proof}

\begin{corollary}\cite[Theorem III.8]{sher2011}
  If $X$ has asymptotic property~C and $Y$ has finite asymptotic dimension, then $X \times Y$ has asymptotic property~C.
\end{corollary}

\subsection{Free product of metric spaces}

\begin{definition}
  Let $(X, x_0)$ be a metric space with a fixed base point~$x_0$.
  The \df{free product} $\ast X$ is the metric space whose
elements are words in the alphabet $X \setminus \{ x_0 \}$ along with the trivial word $\epsilon$. 

We denote $X\setminus \{x_0\}$ by $X^\ast$. We identify elements of $X^\ast$ with the set of single-letter (nontrivial) words in $\ast X$.
  
There is a natural notion of \df{concatenation} of words, which we usually denote by juxtaposition. When we want to emphasize the concatenation of two elements $u$ and $v$ in $\ast X$ we write $u\cdot v$. If $u\in \ast X$ and $A\subset \ast X$, we write $u\cdot A$ for the set of all elements of $\ast X$ that are obtained by concatenating an element of $A$ on the left of the element $u\in \ast X$. The notation $X^m$ is used to denote those elements of $\ast X$ that are obtained by concatenating exactly $m$ elements of $X^\ast$. 
We define $X^0=\{\epsilon\}$ and use the notation $X^{\leq m}$ for $\bigcup_{i=0}^m X^i$. Finally, if $A$ and $B$ are subsets of $\ast X$ then, we define their \df{concatenation} $A\cdot B$ to be all elements of the form $a\cdot b$ with $a\in A$ and $b\in B$.

  We define the \df{norm} of the trivial word to be $0$ and the \df{norm} of a letter $x \in X^\ast$ 
to be $\| x \| = \dist(x_0, x)$.
  The \df{norm} of a composite word is the sum of norms of its letters: $\| x_1 x_2 \cdots x_k \| = \sum_{i = 1}^k \| x_i \|$.
  The \df{distance} between two elements of $\ast X$ is determined by eliminating their common prefix: if $w = u \cdot x \cdot v$ and $w' = u \cdot x' \cdot v'$ (with $u, v, v' \in \ast X$ and $x, x' \in X^\ast$, 
$x \neq x'$), then $\dist(w, w') = \dist(x, x') + \| v \| + \| v' \|$. The distance from $w$ to the trivial word $\epsilon$ is equal to the norm $\| w \|$ of $w$.
\end{definition}

\begin{comm}
  For simplicity of proofs we defined the \emph{unary} free product of a \emph{single} metric space $X$. For applications, we use the fact that if $X$ and $Y$ are metric spaces, then $X \ast Y$ embeds isometrically into $\ast (X \vee Y)$, where $X \vee Y$ denotes the wedge of $X$ and $Y$ and $X \ast Y$ denotes the free product of $X$ and $Y$ defined as in~\cite{bell-dranishnikov2001}.
\end{comm}

\begin{definition}
  A subset $A \subset \ast X$ is said to be \df{flat} if there exists some $x \in \ast X$ such that $A \subset x \cdot (X^\ast)$. 
\end{definition}

\begin{definition}
Let $R>0$. The \df{$R$-cone} of a set $A \subset \ast X$ is defined to be
  \[
    \con_R A = A \cdot (\ast \bar B(x_0, R)) \subset \ast X,
  \]
  where $\bar B(x_0, R)$ denotes the closed ball with center $x_0$ and radius $R$.
\end{definition}

\begin{definition}\label{def:discrete}
We will say that the metric space $X$ is \df{discrete in the metric sense} if $\inf\{\dist(x,y)\colon x,y\in X, x\neq y\}>0$.
\end{definition}

\begin{lemma}\label{lem:cone family}
  Let $X$ be a metric space that is discrete in the metric sense. If $\{ A_\alpha \}$ is a family of flat uniformly bounded subsets of $\ast X$, then for each $M$, the family $\{ \con_M A_\alpha \}$ has asymptotic dimension bounded by $1$ uniformly.
\end{lemma}
\begin{proof}
  Let $E = \inf \{ \dist(x, y) \colon x,y \in X, x \neq y \}$.
  By the assumption, $E > 0$.
  Let $D = \sup_\alpha \{ \diam(A_\alpha) \}$.
  By the assumption, $D < \infty$.
  
  For each $\alpha$ we define a simplicial tree $T_\alpha$.
  The set of vertices of $T_\alpha$ consists of elements of $\con_M A_\alpha$ and of an additional element $r_\alpha$, which we call the root vertex.
  The root vertex $r_\alpha$ is connected by edges to all vertices from base $A_\alpha$ of the cone $\con_M A_\alpha$.
  Two vertices $p, t$ of the cone $\con_M A_\alpha$ are connected by an edge if and only if there exists $x \in \bar B(x_0, M)$, $x \neq x_0$, such that $p = t\cdot x$ or $t = p\cdot x$.
  The metric $\dist_{T_\alpha}$ on $T_\alpha$ is the geodesic metric with edge lengths equal to $1$.
  
  Consider the map $f_\alpha \colon \con_M A_\alpha \to T_\alpha$ that maps elements of $\con_M A$ to their corresponding vertices of $T_\alpha$. We will show that $f_\alpha$ satisfies the following inequalities
\[
    \frac{\dist(x,y)}{M} - \frac{D}{M} \leq \dist_{T_\alpha}(f_\alpha(x), f_\alpha(y)) \leq \frac{\dist(x,y)}{E} + 3\hbox{,}
\]
for each $x,y \in \con_M A_\alpha$.

  As a first case we consider $x,y \in \con_M A_\alpha$ such that $x = a x_1 x_2 \cdots x_k$ and $y = b y_1 y_2 \cdots y_l$ where $a, b \in A_\alpha$, $a \neq b$ and $x_i, y_i \in \bar B(x_0, R)^\ast$, $k, l \geq 0$. We have
  \[
    \dist(x, y) = \left(\sum_{i=1}^{k} \| x_i \| \right)
      + \dist(a, b) + \left(\sum_{i=1}^{l} \| y_i \| \right)
  \]
  and
  \[
    \dist_{T_\alpha}(f_\alpha(x), f_\alpha(y)) = 2 + k + l.
  \]
  We have
  \[
    E \leq \dist(a,b) \leq D, \ \ \ 
    E \leq \| x_i \| \leq M,\ \ \ 
    E \leq \| y_i \| \leq M.
  \]
  Therefore
  \[
    (k + l + 1) E \leq \dist(x,y) \leq (k + l) M + D\hbox{.}
  \]
  
    As a second case consider $x, y \in \con_M A_\alpha$ such that $x = c \cdot x_1 x_2 \cdots x_k$, $y = c \cdot y_1 y_2 \cdot y_l$, $c \in \con_M A_\alpha$, $x_i, y_i \in \bar B(x_0, R)^\ast$, $x_1 \neq y_1$.
  We have
  \[
    \dist(x, y) = \left( \sum_{i = 2}^{k} \| x_i \| \right) +
       \dist(x_1, y_1) + \left( \sum_{i=2}^l \| y_i \| \right)
  \]
  and
  \[
    \dist_{T_\alpha}(f_\alpha(x), f_\alpha(y)) = k + l.
  \]
  We have
  \[
    E \leq \dist(x_1, y_1) \leq D, \ \ \ 
    E \leq \| x_i \| \leq M,\ \ \ 
    E \leq \| y_i \| \leq M.
  \]
  Therefore
  \[
   (k + l - 1) E \leq \dist(x,y) \leq (k + l) M\hbox{.}
  \]
  
  Combining inequalities from both cases we have
  \[
    \dist_{T_\alpha}(f_\alpha(x), f_\alpha(y)) \leq k + l + 2
      \leq \frac{\dist(x,y)}{E} + 3
  \]
  and
  \[
   \frac{\dist(x,y)}{M} - \frac{D}{M} \leq k + l \leq \dist_{T_\alpha}(f_\alpha(x), f_\alpha(y))\hbox{.}
  \]
  
  Hence the inequalities that we sought:
  \[
    \frac{\dist(x,y)}{M} - \frac{D}{M} \leq \dist_{T_\alpha}(f_\alpha(x), f_\alpha(y)) \leq \frac{\dist(x,y)}{E} + 3.
  \] 
  Therefore each $f_\alpha$ is a quasi-isometry with constants that do not depend on $\alpha$.

In the example following Theorem 19 of~\cite{b-d-asymptotic-dimension}, it is shown that for a given $r>0$ one can take a cover of any simplicial tree by $3r$-bounded $r$-disjoint families of subsets. It follows that any family of simplicial trees has asymptotic dimension bounded by $1$ uniformly. 
  Since the family $\{ \con_M A_\alpha \}$ is uniformly quasi-isometric to a family of simplicial trees, it has asymptotic dimension bounded by $1$ uniformly.
\end{proof}

\begin{definition}
  We say that a set $A$ is \df{$R$-connected} if for every $x, y \in A$ there exists a finite sequence $x_0 = x, x_1, x_2, \ldots, x_k = y$ such that $\dist(x_i, x_{i+1}) \leq R$ for each $i = 1, 2, \ldots, k-1$. An \df{$R$-connected component} of $A$ (or simply an \df{$R$-component}) is a maximal $R$-connected subset of $A$.
  We recall that in Section \ref{sec:preliminaries} we described the sets $A, B$ as being \df{$R$-disjoint} if for each $x \in A$, $y \in B$ we have $\dist(x,y) > R$.
\end{definition}


\begin{definition}
  The \df{order} of an element $x$ of $\ast X$ is the non-negative integer $m$ such that $x \in X^m$. We write $\ord(x)=m$ in this case. 
  \end{definition}


\begin{lemma}\label{lem:components}
  Let $X$ be a discrete metric space and let $R > 0$.
  Let $A \subset \ast X$ be such that $A \subset \ast X \cdot (A \setminus \bar B(x_0, R))$.
  Assume that $R$-connected components of $A$ are uniformly bounded.
  Then for each $M$, $R$-connected components of $\con_M A$ have asymptotic dimension bounded by $1$ uniformly.
\end{lemma}
\begin{proof}
  As a first step we will show that 
\begin{quote}
$C$ is an $R$-connected component of $\con_M A$
  \[ \Updownarrow \]
$C \cap X^{\leq n}$ is an $R$-connected component of $\con_M A \cap X^{\leq n}$ for each $n$
\end{quote}
  
  The implication $\Uparrow$ is trivial.
  
  We will say that two words $w_1, w_2$ in $\ast X$ are \df{adjacent}  if 
  \begin{enumerate}
  \item $w_1 = w_2 \cdot x$ for some $x \in X$ or
  \item $w_2 = w_1 \cdot x$ for some $x \in X$ or
  \item $w_1 = w \cdot x_1$ and $w_2 = w \cdot x_2$ for some $w \in \ast X$ and $x_1, x_2 \in X$.
  \end{enumerate}
  In other words $w_1$ and $w_2$ are adjacent if they differ by their last letter or if one is equal to the other with a single letter appended at the end.
  
  Let $x = p \cdot x_1 x_2 \cdot x_k$, $y = p \cdot y_1 y_2 \cdot y_l$, $x, y \in \con_M A$, $\dist(x,y) \leq R$, $p$ is the common prefix of $x$ and $y$. We have
  \[
    \dist(x, y) = \left(\sum_{i = 2}^{k} \| x_i \| \right)
    	      + \dist(x_1,y_1) +
              \left(\sum_{i = 2}^{l} \| y_i\| \right)
         \leq R
  \]
  Therefore $x_i, y_i \in \bar B(x_0, R)$ for $i \geq 2$ so $x_i, y_i \not\in A$ for $i \geq 2$. Hence $p \cdot x_1 x_2 \cdots x_i \in \con_M A$ and $p \cdot y_1 y_2 \cdots y_i \in \con_M A$ for each $i$.
  Therefore $x = p \cdot x_1 x_2 \cdots x_k$, $p \cdot x_1 x_2 \cdots x_{k-1}$, $\ldots$, $p \cdot x_1$, $p \cdot y_1$, $p \cdot y_1 y_2$, $\ldots$, $p \cdot y_1 y_2 \cdots y_l = y$ is a sequence of $R$-close adjacent elements of $\con_M A$.
  
  Therefore if $C$ is an $R$-connected component of $\con_M A$, then for each $x, y \in C$ there exists a sequence $x = x_0, x_1, \ldots, x_k = y$ of adjacent $R$-close points of $C$.

  Observe that if $x_1, x_2, \ldots, x_k$ is a sequence of adjacent points, if $\ord(x_1) = \ord(x_k)$, and if $\ord(x_i) > \ord(x_1)$ for all $i$ with $2 \leq i \leq k-1$, then $x_1 = x_k$. This implies that if $x, y \in \con_M A$ are in the same $R$-connected component and $x, y \in X^{\leq n}$, then $x,y$ are in the same $R$-connected component of $C \cap X^{\leq n}$. The implication $\Downarrow$ is proven.
  
Let $D$ be a bound on the diameters of the $R$-connected components of $A$.
Let $C$ be an $R$-connected component of $\con_M A$.
Let $C_k = C \cap X^{\leq k}$.
We have proven that each $C_k$ is $R$-connected.
Let $k_0$ be the smallest integer such that $C_{k_0} \neq \emptyset$.
The diameter of $C_{k_0}$ is bounded by $2M + 2R + 2D$.
We will show by induction on $k$ that
\[
  C_k \subset \con_{M+R+D} C_{k_0}.
\]

Consider $x \in \left( C_{k+1} \setminus C_k \right)$.
Since $C \subset \con_M A$, we have $x \in A$ or $x \in \con_M C_k$.

Let $x\in A$. Then $x$ has to be in an $R$-connected component of $A$ that is $R$-close to $\con_M C_k$. Therefore $x \in \con_{M+R+D} C_k$.

Let $x \in \con_M C_k$. Then $\con_M C_k \subset \con_{M+R+D} C_k$.
Hence $x \in \con_{M+R+D} C_k \subset \con_{M+R+D} \con_{M+R+D} C_{k_0} = \con_{M+R+D} C_{k_0}$.

We proved by induction that if $C$ is an $R$-connected component of $\con_M A$, then $C \subset \con_{M+R+D} C_{k_0}$. Observe that $C_{k_0}$ is flat and has uniformly bounded diameter (by $2M+2R+2D$). Therefore by Lemma~\ref{lem:cone family}, the collection of $R$-connected components of $\con_M A$ has asymptotic dimension bounded by $1$ uniformly.
\end{proof}

\begin{theorem}
  \label{thm:free product x}
  Let $X$ be a discrete metric space with a fixed base point $x_0$.
  If $X$ has asymptotic property~C, then $\ast X$ has asymptotic property~C.  
\end{theorem}
\begin{proof}
  Let $R_1 \leq R_2 \leq \ldots$ be a non-decreasing sequence of real numbers.
  Let $\U{1}, \U{2}, \ldots, \U{n}$ be a finite sequence of $R_i$-disjoint uniformly bounded families of subsets of $X$ such that $\bigcup \U{i}$ covers $X$.
  For $i=1,2,\ldots, n$, let 
  \[
    \V{i} = \{ x \cdot (U \setminus \bar B(x_0, R_{n+1})) \colon U \in \U{i}, x \in \ast X \}
  \]
  and
  \[
    \V{n+1} = \{ \{ x_0 \} \}.
  \]
  We will prove that
  \begin{enumerate}
  \item $\bigcup_{i = 1}^{n+1} (\con_{R_{n+1}} \bigcup \V{i}) = \ast X$.
  \item for each $i$, $\V{i}$ is $R_{i}$-disjoint, uniformly bounded and its elements are flat.
  \end{enumerate}
  
  To prove (1), consider an element $x = x_1 x_2 \cdots x_k \in \ast X$. If for each $i$, $\| x_i \| \leq R_{n+1}$, then $x \in \con_{R_{n+1}} \{ x_0 \}$.
Otherwise let $m$ be the largest integer such that $\| x_m \| > R_{n+1}$. Since $\bigcup \U{i}$ covers $X$, there exists an $\ell$ and a $U \in \U{\ell}$ such that $x_m \in U$. We have $x \in x_1 x_2 \cdots x_{m-1} \cdot (U \setminus \bar B(x_0, R_{n+1})) \cdot \ast \bar B(x_0, R_{n+1})$; hence $x \in \con_{R_{n+1}} \bigcup \V{l}$.

To prove (2), assume that $i$ is given, $1\le i\le n$. Observe that elements of $\V{i}$ are flat and uniformly bounded directly from the definition. So, it remains only to show that $\V{i}$ is $R_i$-disjoint. Suppose that $V_1$ and $V_2$ are distinct elements of $\V{i}$, say $V_1 = x_1 \cdot U_1$ and $V_2 = x_2 \cdot U_2$. We consider two cases: a) if $x_1=x_2$, then the result follows from the fact that the family $\U{i}$ is $R_i$-disjoint; b) if $V_1 \cup V_2$ is not flat, then we take arbitrary elements $v_1=x_1u_1\in V_1=x_1\cdot U_1$ and $v_2=x_2u_2\in V_2=x_2\cdot U_2$. Then, $\dist(v_1, v_2) \geq \min\{\|u_1\|,\|u_2\|\} > R_{n+1} \geq R_i$, by the definition of $\V{i}$ and the definition of the metric in $\ast X$.
  
  By Lemma~\ref{lem:components}, the  $R_i$-connected components of $\con_{R_{n+1} + 1} \bigcup \V{i}$ have asymptotic dimension bounded by $1$ uniformly and are $R_i$-disjoint. By the Decomposition Lemma, $\ast X$ has asymptotic property~C.
\end{proof}

\begin{corollary}
  Let $X, Y$ be metric spaces with fixed base points.
  If $X$ and $Y$ have asymptotic property~C, then $X * Y$ has asymptotic property~C.
\end{corollary}
\begin{proof}
  By the Union Theorem (Proposition \ref{prop:apc-union}), $X \vee Y$ has asymptotic property~C.
  By Theorem~\ref{thm:free product x}, $(X \vee Y) * (X \vee Y)$
    has asymptotic property~C.
  Since $X * Y$ embeds isometrically into $(X \vee Y) * (X \vee Y)$
    and asymptotic property~C is inherited by subspaces, we are done.
\end{proof}


\section{Open questions}

Although we were able to resolve several questions related to asymptotic property~C, we would like to mention the following problems, which remain open.

We were able to show that the free product of two metric spaces (or groups) with asymptotic property~C has asymptotic property~C, but this invites the obvious question for amalgamated products~\cite[Questions 4.3]{bell-moran2015}: Is asymptotic property~C preserved by amalgamated free products? Generalizing this a bit more, one could consider the following question.

Let $\Gamma$ be a finite graph. To each vertex $v$ of $\Gamma$ we assign a group $G_v$. We let $\mathfrak{G}\Gamma$ denote the graph product of these groups; i.e., the group generated by the $G_v$ along with the additional relation that two elements of the groups $G_v$ and $G_w$ commute precisely when an edge of $\Gamma$ connects $v$ and $w$. If all the $G_v$ have asymptotic property~C, does $\mathfrak{G}\Gamma$ have asymptotic property $C$?

These questions and the question of whether a semi-direct product preserves asymptotic property~C could be addressed via a true fibering theorem. We state this in two forms, both as a statement for metric spaces and a statement in terms of exact sequences of groups.

\begin{question}
Let $f\colon X\to Y$ be a uniformly expansive map between metric spaces. Assume that $Y$ has asymptotic property~C and that $f^{-1}(A)$ has asymptotic property~C for every bounded subset $A\subset Y$. Does $X$ have to have asymptotic property~C?
\end{question}

\begin{question}\label{q:apc fiber}\ 
  Let $f \colon G \to H$ be a group homomorphism such that $H$ and $\ker f$ have asymptotic property~C. Does $G$ have to have asymptotic property~C?
\end{question}

One could also ask several questions related to the construction of Yamauchi \cite{yamauchi2015}. In particular, one could ask~\cite[Question 3.2]{yamauchi2015}: Does every countable direct union of groups with finite asymptotic dimension have asymptotic property~C? More generally, is asymptotic property~C preserved by countable direct unions?


Dydak and Virk showed that the notions of countable asymptotic dimension and straight finite decomposition complexity are equivalent~\cite{dydak-virk2016}. It is easy to see that asymptotic property~C implies either of these two notions, but the exact nature of the relationship between these notions is unknown for groups. We believe that the asymptotic property~C can be generalized to a property we propose to call $\omega$-APC. This generalization enables us to see that our fibering lemma and decomposition lemma are both special cases of a more general framework of which straight finite decomposition complexity and asymptotic property~C are the first nontrivial special cases.

\bibliography{references}

\begin{thebibliography}{10}

\bibitem{beckhardt2015}
S.~Beckhardt.
\newblock Groups acting on metric spaces with asymptotic property {C}.

\bibitem{beckhardt-goldfarb2016}
S.~Beckhardt and B.~Goldfarb.
\newblock Extension properties of asymptotic property {C} and finite
  decomposition complexity.

\bibitem{bell-dranishnikov2001}
G.~Bell and A.~Dranishnikov.
\newblock On asymptotic dimension of groups.
\newblock {\em Algebr. Geom. Topol.}, 1:57--71, 2001.

\bibitem{b-d-asymptotic-dimension}
G.~Bell and A.~Dranishnikov.
\newblock Asymptotic dimension.
\newblock {\em Topology Appl.}, 155(12):1265--1296, 2008.

\bibitem{bell-moran2015}
G.~Bell and D.~Moran.
\newblock On constructions preserving the asymptotic topology of metric spaces.
\newblock {\em North Carolina Journal of Mathematics and Statistics}, 1:46--57,
  2015.

\bibitem{bell-moran-nagorko2016}
G.~Bell, D.~Moran, and A.~Nag\'orko.
\newblock Coarse property {C} and decomposition complexity.
\newblock {\em Topol. Appl.}, (in press) 2016.

\bibitem{bell2003}
G.~C. Bell.
\newblock Property {A} for groups acting on metric spaces.
\newblock {\em Topology Appl.}, 130(3):239--251, 2003.

\bibitem{b-d-hurewicz}
G.~C. Bell and A.~N. Dranishnikov.
\newblock A {H}urewicz-type theorem for asymptotic dimension and applications
  to geometric group theory.
\newblock {\em Trans. Amer. Math. Soc.}, 358(11):4749--4764, 2006.

\bibitem{BDLM}
N.~Brodskiy, J.~Dydak, M.~Levin, and A.~Mitra.
\newblock A {H}urewicz theorem for the {A}ssouad-{N}agata dimension.
\newblock {\em J. Lond. Math. Soc. (2)}, 77(3):741--756, 2008.

\bibitem{Davila}
T.~{Davila}.
\newblock {On asymptotic property C}.
\newblock {\em ArXiv e-prints}, Nov. 2016.

\bibitem{dranishnikov-smith-2006}
A.~Dranishnikov and J.~Smith.
\newblock Asymptotic dimension of discrete groups.
\newblock {\em Fund. Math.}, 189(1):27--34, 2006.

\bibitem{dranishnikov-zarichnyi2014}
A.~Dranishnikov and M.~Zarichnyi.
\newblock Asymptotic dimension, decomposition complexity, and {H}aver's
  property {C}.
\newblock {\em Topology Appl.}, 169:99--107, 2014.

\bibitem{dranish2000}
A.~N. Dranishnikov.
\newblock Asymptotic topology.
\newblock {\em Uspekhi Mat. Nauk}, 55(6(336)):71--116, 2000.

\bibitem{dydak2016}
J.~Dydak.
\newblock Coarse amenability and discreteness.
\newblock {\em J. Aust. Math. Soc.}, 100(1):65--77, 2016.

\bibitem{dydak-virk2016}
J.~Dydak and {\v{Z}}.~Virk.
\newblock Preserving coarse properties.
\newblock {\em Rev. Mat. Complut.}, 29(1):191--206, 2016.

\bibitem{engelking1995}
R.~Engelking.
\newblock {\em Theory of dimensions finite and infinite}, volume~10 of {\em
  Sigma Series in Pure Mathematics}.
\newblock Heldermann Verlag, Lemgo, 1995.

\bibitem{grave2006}
B.~Grave.
\newblock Asymptotic dimension of coarse spaces.
\newblock {\em New York J. Math.}, 12:249--256 (electronic), 2006.

\bibitem{gromov93}
M.~Gromov.
\newblock Asymptotic invariants of infinite groups.
\newblock In {\em Geometric group theory, {V}ol.\ 2 ({S}ussex, 1991)}, volume
  182 of {\em London Math. Soc. Lecture Note Ser.}, pages 1--295. Cambridge
  Univ. Press, Cambridge, 1993.

\bibitem{guentner2014}
E.~Guentner.
\newblock Permanence in coarse geometry.
\newblock In {\em Recent progress in general topology. {III}}, pages 507--533.
  Atlantis Press, Paris, 2014.

\bibitem{hattori-yamada1989}
Y.~Hattori and K.~Yamada.
\newblock Closed pre-images of {$C$}-spaces.
\newblock {\em Math. Japon.}, 34(4):555--561, 1989.

\bibitem{haver}
W.~E. Haver.
\newblock A covering property for metric spaces.
\newblock In {\em Topology {C}onference ({V}irginia {P}olytech. {I}nst. and
  {S}tate {U}niv., {B}lacksburg, {V}a., 1973)}, pages 108--113. Lecture Notes
  in Math., Vol. 375. Springer, Berlin, 1974.

\bibitem{nowak2007}
P.~W. Nowak.
\newblock Coarsely embeddable metric spaces without {P}roperty {A}.
\newblock {\em J. Funct. Anal.}, 252(1):126--136, 2007.

\bibitem{openproblems}
E.~Pearl, editor.
\newblock {\em Open problems in topology. {II}}.
\newblock Elsevier B. V., Amsterdam, 2007.

\bibitem{EPol1986}
E.~Pol.
\newblock A weakly infinite-dimensional space whose product with the
  irrationals is strongly infinite-dimensional.
\newblock {\em Proc. Amer. Math. Soc.}, 98(2):349--352, 1986.

\bibitem{Pol-Pol2009}
E.~Pol and R.~Pol.
\newblock A metric space with the {H}aver property whose square fails this
  property.
\newblock {\em Proc. Amer. Math. Soc.}, 137(2):745--750, 2009.

\bibitem{roe2003}
J.~Roe.
\newblock {\em Lectures on coarse geometry}, volume~31 of {\em University
  Lecture Series}.
\newblock American Mathematical Society, Providence, RI, 2003.

\bibitem{rohm1990}
D.~M. Rohm.
\newblock Products of infinite-dimensional spaces.
\newblock {\em Proc. Amer. Math. Soc.}, 108(4):1019--1023, 1990.

\bibitem{serre}
J.-P. Serre.
\newblock {\em Trees}.
\newblock Springer Monographs in Mathematics. Springer-Verlag, Berlin, 2003.
\newblock Translated from the French original by John Stillwell, Corrected 2nd
  printing of the 1980 English translation.

\bibitem{sher2011}
L.~Sher.
\newblock Asymptotic dimension and asymptotic property {C}.
\newblock Master's thesis, The University of North Carolina at Greensboro,
  2011.

\bibitem{yamauchi2015}
T.~Yamauchi.
\newblock Asymptotic property {C} of the countable direct sum of the integers.
\newblock {\em Topology Appl.}, 184:50--53, 2015.

\end{thebibliography}

\end{document}